\documentclass[dvipsnames]{article}
\usepackage[portrait,margin=1in]{geometry}
\usepackage[utf8]{inputenc}

\usepackage{mathtools,amssymb,amsthm}

\usepackage{nicematrix}

\usepackage{tikz}
   \usetikzlibrary{arrows,positioning}
   \usetikzlibrary{decorations.pathreplacing}

\usepackage{hyperref}
\hypersetup{
  colorlinks   = true, %Colours links instead of ugly boxes
  urlcolor     = blue, %Colour for external hyperlinks
  linkcolor    = blue, %Colour of internal links
  citecolor    = blue  %Colour of citations
}                                           %

\usepackage[capitalise,noabbrev]{cleveref}

\theoremstyle{plain}
   \newtheorem{theorem}{Theorem}[section]
   \newtheorem{lemma}[theorem]{Lemma}
   
   \newtheorem{corollary}[theorem]{Corollary}
   
\theoremstyle{definition}
   \newtheorem{definition}[theorem]{Definition}
   \newtheorem{example}[theorem]{Example}
   \newtheorem{remark}[theorem]{Remark}

\DeclareMathOperator{\MLdeg}{MLdeg}

\title{Maximum likelihood degree of the $\beta$-stochastic blockmodel}
\author{Cashous Bortner\footnote{California State University, Stanislaus}, Jennifer Garbett\footnote{Lenoir-Rhyne University}, Elizabeth Gross\footnote{University of Hawai`i at M\={a}noa}, Christopher McClain\footnote{West Virginia University Institute of Technology},\\ Naomi Krawzik\footnote{Sam Houston State University}, Derek Young\footnote{Mount Holyoke College}}

\usepackage[mathlines]{lineno} 
\usepackage{soul}

\begin{document}

\maketitle
\begin{abstract}
    Log-linear exponential random graph models are a specific class of statistical network models that have a log-linear representation. This class includes many stochastic blockmodel variants. In this paper, we focus on $\beta$-stochastic blockmodels, which combine the $\beta$-model with a stochastic blockmodel. Here, using recent results by Almendra-Hern\'{a}ndez, De Loera, and Petrovi\'{c}, which describe a Markov basis for $\beta$-stochastic block model, we give a closed form formula for the maximum likelihood degree of a $\beta$-stochastic blockmodel. The maximum likelihood degree is the number of complex solutions to the likelihood equations.  In the case of the $\beta$-stochastic blockmodel, the maximum likelihood degree factors into a product of Eulerian numbers.
\end{abstract}

\section{Introduction}

\emph{Log-linear models} form a popular class of statistical models used in categorical data analysis that have been studied extensively in algebraic statistics (see for example \cite[Chapter 1]{DSS2009} and \cite[Chapter 9]{sullivant2018}). In part, what makes log-linear models so amenable to algebraic techniques is that they have a monomial parameterization, and thus correspond to toric varieties. Here, we focus on a particular class of log-linear models, namely \emph{log-linear exponential random graph models} (log-linear ERGMs) \cite{gross2024goodnessfitloglinearergms}, \cite[Chapter 11]{sullivant2018}, \cite{gross2021algebraic}. Log-linear ERGMs are statistical network models used to describe relational data (e.g.~when data is in the form of a graph or network). They are exponential families of probability distributions over the space of simple graphs (directed or undirected, depending on model) with $n$ vertices where, in the most common setting, the sufficient statistic is a linear function of the adjacency matrix.

Log-linear ERGMs can be viewed as generative models where different effects govern edge formation. For example, Erd\H{o}s-R\'{e}nyi random graphs are one of the simplest log-linear ERGMs where density is the only governing effect modeled; in this setting, density is controlled with a single parameter, and the corresponding sufficient statistic is the total number of edges. The next simplest log-linear ERGM, the $\beta$-model, includes a parameter for each node that models the expansiveness for each edge, i.e., the propensity to be connected to others. The sufficient statistic for the $\beta$-model is the degree sequence of the network.  Stochastic blockmodels, statistical network models where vertices are placed into groups, or blocks, and the parameters of the model govern within-block and between-block density, also belong to the class of log-linear ERGMs. These models are particularly relevant when modeling homophily, a tendency for vertices with similar attributes to be connected. In this paper, we focus on the subclass of log-linear ERGMs called $\beta$-stochastic blockmodels ($\beta$-SBM), which combines the $\beta$-model with a stochastic blockmodel \cite{Yang2015}, \cite{KPPSARWWY2024}, \cite{gross2024goodnessfitloglinearergms}.

Log-linear ERGMs have been studied in algebraic statistics from multiple angles. The geometry of maximum likelihood estimation for the $\beta$-model and variants are studied in \cite{rinaldo2013} and \cite{stasi2013}. Combinatorial and algebraic methods for goodness-of-fit testing for log-linear ERGMs are studied in \cite{gross2024goodnessfitloglinearergms} and \cite{KPPSARWWY2024}, while their Markov bases have been studied in \cite{ADP2024}. Indeed, the recent paper \cite{ADP2024} shows that the $\beta$-SBM has a quadratic Markov basis, which is a fundamental fact that we draw upon in this work. Here, we continue the investigation of the $\beta$-SBM from an algebraic statistics framework by investigating its maximum likelihood degree. 

The \emph{maximum likelihood degree} (ML degree) of a statistical model is the number of complex solutions to the likelihood equations for a generic data point and is a measure of the algebraic complexity of maximum likelihood estimation \cite{CHKS2006}, \cite{huh2014likelihood}. While statistics is the main application of the ML degree, it is defined with respect to a complex algebraic variety. Indeed, the ML degree considers the optimization problem of likelihood estimation over the \emph{Zariski closure} of the statistical model, thus it can be also defined for an arbitrary algebraic variety. The Zariski closure of a log-linear model, such as the $\beta$-SBM, is a toric variety. Recently, there has been a range of work focusing on the ML degrees of different families of toric varieties. For example, the ML degrees for scaled toric varieties are studied in \cite{ABBGHHNRS2019}, where the authors compute the ML degrees of rational normal scrolls and a large class of Veronese-type varieties. The focus on the ML degrees of toric varieties continues in the literature with: \cite{amendola2024likelihood}, which studies the ML degree when the design matrix corresponds to a reflexive polytope; with \cite{amendola2020maximum}, which studies the ML degree of 2-dimensional Gorenstein toric Fano varieties; with \cite{coons2021quasi}, which classifies the two-way quasi-independence models with ML degree equal to one; with \cite{clarke2023matroid}, which explores how the ML degree drops under different scalings of independence models and models defined by the second hypersimplex; with \cite{brysiewicz2023lawrence}, which studies the ML degrees of hierarchical models and three dimensional quasi-independence models; and with \cite{duarte2021discrete}, which studies the ML degree for staged tree models. This work fits within this body of literature by giving a closed form formula for the ML degree for another family of toric varieties, one coming from statistical network analysis.

The main theorem (Theorem \ref{thm:RecursionTheorem})  of this paper gives a multiplicative formula for the ML degree for a $\beta$-SBM with $k$ blocks with $n_1, \ldots, n_k$ vertices in each block. In particular, if $k>1$ and $N\subseteq [k]=\{1, \ldots, k\}$ is the set of indices for blocks containing more than two vertices, the ML degree for this $\beta$-SBM is 
\[
\MLdeg(n_1, \ldots, n_k) = \prod_{i\in N} (2^{n_i} - n_i-1)
\]
when $N$ is non-empty, and $\MLdeg(n_1,\dots,n_k)=1$ otherwise, i.e. when all blocks have size 1 or 2. Notice that each factor in the product is the Eulerian number $A(n_i, 1)$. When there is only a single block, the $\beta$-SBM collapses to the log-linear model corresponding to the second hypersimplex $\Delta_{2, n_1}$. The ML degree for this case is the Eulerian number $A(n_1 - 1, 1)$ (see Remark 23 in \cite{ABBGHHNRS2019}).

The paper is structured as follows: In \cref{sec:Background} we provide necessary background on the $\beta$-SBM, including the design matrix, the corresponding toric ideal, and the likelihood equations. In \cref{sec:MainThm}, we state the main result, followed by its proof in \cref{sec:MainProof}. The proof relies on the fact that the ML degree factors into the product of the ML degrees of two submodels (\cref{lem:FactoringLemma}); we prove this result in \cref{sec:factoringproof}.

\section{Background}\label{sec:Background}

\subsection{The \texorpdfstring{Log-linear ERGMS and the $\beta$}{beta}-Stochastic Blockmodel} 

In general, an exponential random graph model is a collection of probability distributions on the space of all graphs  on $n$ vertices $\mathcal G_n$ \footnote{Depending on context, sometimes $\mathcal G_n$ represents the space of all directed  graphs on $n$ vertices and sometimes it represents the space of all  undirected graphs on $n$ vertices, with possibly other constraints, such as simple, also specified. For us, in the remainder of the manuscript, $\mathcal G_n$ will be the space of all simple, undirected graphs on $n$ vertices. }  with the following form: 
\[
P_{\theta}(G) = Z(\theta) e^{{\theta} \cdot t(G)}, \ G \in \mathcal G_n, \
\]
where $G$ is represented as a vector in $\mathbb{R}^{\ell}$ (where $\ell$ depends on the types of graphs considered), 
$\theta$ is a row vector of parameters of length $q$, $t$ is a map $t: \mathbb{R}^{\ell} \to \mathbb{R}^q$ 
called the \emph{sufficient statistic}, and $Z(\theta)$ is a normalizing constant. The image of the sufficient statistic $t$ is a vector where each entry is a network statistic, e.g.~edge count, degree of a given vertex,  number of edges in a given block of vertices, etc. When the sufficient statistic is a linear function on the entries of a natural contingency table representation $u$ of the graph, as in degree-based models or stochastic blockmodels, the sufficient statistic map $t$ can be described with a design matrix $A$, and $Au$ replaces $t(G)$ in the expression above.
In this case, we call the model a \emph{log-linear ERGM}. This connection between some statistical network models and log-linear models was first established in \cite{fienberg1981categorical} and \cite{fienberg1985statistical}.

The $\beta$-stochastic blockmodel is a log-linear ERGM that combines the features of the beta model and the stochastic blockmodel \cite{Yang2015}, \cite{KPPSARWWY2024}, \cite{gross2024goodnessfitloglinearergms}; it is also known as the exponential family version of the \emph{degree corrected stochastic block model}. The description of the model begins with a set of vertices and a block assignment of the vertices. In this exposition, we assume that the block assignment is known. This case is useful when testing for effects of homophily in networks. In other applications, such as clustering, the block assignment is treated as a latent variable.

Let $n$ and $k$ be integers such that $n\geq 2$ and $k \geq 1$. Here, $n$ will be the total number of vertices and $k$ the number of blocks. Given positive integers $n_1,n_2,\ldots,n_k$ whose sum is $n$, let $V_i=\{(i,v):1\leq v\leq n_i\}$ for $1\leq i\leq k$. The set $V=\bigcup_{i=1}^k V_i$ is the set of vertices, partitioned by the blocks $V_i$ of sizes $n_i$, respectively. Notice that we are defining the vertex set so that each vertex is an ordered pair where the first entry indicates the block membership of the vertex. While initially cumbersome, this notation will have advantages later on. We denote by $E$ the set of \textit{dyads}, i.e., potential (undirected) edges $\{ (i,v), (j,w) \}$. Formally,
\[
E=\bigcup_{i=1}^{k}\bigcup_{j=i}^{k} E_{i,j}, \quad 
\text {where } \quad
E_{i,j} = E_{j,i} =
\begin{cases} 
   \{\{(i,v),(i,w)\}: 1\leq v<w\leq n_i\}                & \text{if } i=j \\
   \{\{ (i,v)(j,w) \}: 1\leq v\leq n_i, 1\leq w\leq n_j\} & \text{if } i\neq j 
\end{cases},
\]
where $E_{i,i}$ contains potential edges within block $i$ and those between distinct blocks $i$ and $j$ are in $E_{i,j}$.

Following \cite{KPPSARWWY2024} and \cite{gross2024goodnessfitloglinearergms}, the $\beta$-SBM is parametrized by node-specific parameters $\beta_{(i,v)}$ for $1\leq i\leq k$ and $1\leq v\leq n_i$, as in the beta model, and block-specific parameters $\alpha_{i,j}$ for $1\leq i\leq j\leq k$ , as in the stochastic blockmodel. The $\beta$-SBM is a \emph{dyad-independent} model, meaning that the presence or absence of each edge is independent of the presence or absence of any other edge. Thus, we can specify the probability of observing a specific graph by specifying the probabilities of each edge. To this end, we give the log-odds for the probability $p_{(i,v)(j,w)}$ of each dyad $(i,v)(j,w)$ being connected by an edge:
\[
\log\left(\frac{p_{(i,v)(j,w)}}{1-p_{(i,v)(j,w)}}\right)=\beta_{(i,v)}+\beta_{(j,w)}+\alpha_{i,j.}
\]
Alternatively, with an appropriate transformation of parameters,  we can specify a monomial parameterization for the probabilities:
\[
p_{(i,v)(j,w)} = \beta'_{(i,v)}\beta'_{(j,w)}\alpha'_{i,j}.
\]
We use $\mathcal{M}(n_1,n_2,\ldots,n_k)$ to denote the set of all probability distributions on $\mathcal G_n$ that arise from the $\beta$-SBM parameterization where $k$ is the number of vertex blocks and
$n_i$ is the size of the $i$th block.

Log-linear ERGMs are defined in terms of their sufficient statistic, which is related to the design matrix $A$. In particular, if we determine the probability of observing a given graph $G \in \mathcal G_n$ by multiplying the probabilities of observing each dyad in the graph, then the exponents on the $\beta'$ and $\alpha'$ will correspond to the entries of the sufficient statistic of $G$ for the model. In this case, the entries of the sufficient statistic will be elements of the degree sequence (e.g.~the exponent on $\beta'_{(1,1)}$ is the degree of vertex $(1,1)$ in $G$) and the numbers of edges within blocks and between block pairs. For a graph $G$ on vertex set $V$, let $d_{(i,v)}$ be the degree of the vertex $(i,v)$ in $G$, i.e., the number of edges in $G$ which are incident with the vertex $(i,v)$. More formally, the vector of sufficient statistics for $G$ is
\begin{equation}\label{eqn:SufficientStatistic}
    t(G)=\left (d_{(1,1)},d_{(1,2)}\dots,d_{(k,n_k)},|E_{1,1}|,|E_{1,2}|,\dots,|E_{1,k}|, |E_{2,2}|, |E_{2,3}|, \ldots, |E_{k,k}|\right )
\end{equation}
in which the vertices are ordered lexicographically and the block pairs are also ordered lexicographically.

\begin{example}
\cref{fig:ModelLabeling} shows a graph $G$ containing $3$ blocks of sizes $3$, $4$, and $3$ with some edges between the blocks. The sufficient statistic $t(G)$ for the graph is  $(2,4,3,4,4,3,2,3,3,2,2,3,2,3,4,1)$.
\end{example}

\begin{figure}[htpb]
\centering
\begin{tikzpicture}[
single/.style={circle,draw=black,fill=black},
block/.style={oval,thick},
scale=0.8,
fill opacity=0.7,
]

\draw (5.0,5.50) node {Block $1$};
\draw[fill=blue,fill opacity = 0.5] (5,4) ellipse (1.7cm and 1.0cm){};
\node[single] at (4.0,3.8) (11){};
\node[] at (4.0,4.3) {$(1,1)$};
\node[single] at (5.0,3.8) (12){};
\node[] at (5.0,4.3) {$(1,2)$};
\node[single] at (6.0,3.8) (13){};
\node[] at (6.0,4.3) {$(1,3)$};

\draw (1.5,2.25) node {Block $2$};
\draw[fill=green,fill opacity = 0.5] (1.5,0) ellipse (1.3cm and 1.7cm){};
\node[single] at (1.8,1.0) (21){};
\node[] at (1.0,1.0) {$(2,1)$};
\node[single] at (1.8,0.4) (22){};
\node[] at (1.0,0.4) {$(2,2)$};
\node[single] at (1.8,-0.2) (23){};
\node[] at (1.0,-0.2) {$(2,3)$};
\node[single] at (1.8,-0.8) (24){};
\node[] at (1.0,-0.8) {$(2,4)$};

\draw (8.5,2.25) node {Block $3$};
\draw[fill=red,fill opacity = 0.5] (8.5,0) ellipse (1.3cm and 1.7cm){};
\node[single] at (8.2,0.8) (31){};
\node[] at (9.0,0.8) {$(3,1)$};
\node[single] at (8.2,0.0) (32){};
\node[] at (9.0,0.0) {$(3,2)$};
\node[single] at (8.2,-0.8) (33){};
\node[] at (9.0,-0.8) {$(3,3)$};

\draw[-] (11) -- (12) -- (13);
\draw[-] (12) -- (33);
\draw[-] (13) -- (31);

\draw[-] (21) -- (11);
\draw[-] (21) -- (22);
\draw[-] (21) -- (12);
\draw[-] (21) -- (32);

\draw[-] (22) -- (13);
\draw[-] (22) -- (23);
\draw[-] (22) -- (32);

\draw[-] (23) -- (24);
\draw[-] (23) -- (33);

\draw[-] (24) -- (31);

\draw[-] (31) -- (32);

\end{tikzpicture}
\caption{A graph with three blocks of sizes $3,4,$ and $3$.}%
\label{fig:ModelLabeling}
\end{figure}
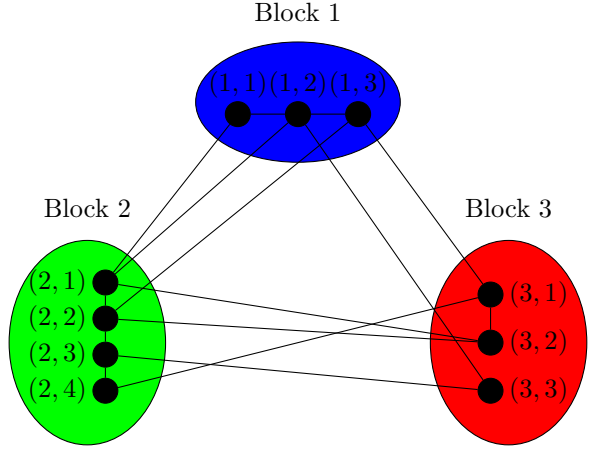

\subsection{The Design Matrix and Likelihood Equations}

As described in \cite{gross2024goodnessfitloglinearergms} and \cite{KPPSARWWY2024}, the $\beta$-SBM can be specified as a log-linear model. The contingency table representation of the model including table dimensions and marginals is given in \cite{gross2024goodnessfitloglinearergms}. For a log-linear model, the defining marginals, which correspond to the sufficient statistic, are described by a design matrix $A$. We begin with describing the design matrix for the $\beta$-SBM.

The \textit{design matrix} $A$ associated with the $\beta$-stochastic blockmodel $\mathcal{M}(n_1,n_2,\ldots , n_k)$ is a $\big( n + k +  \binom{k}{2} \big) \times |E|$ zero-one matrix with entries as follows. We index the columns by the variables $p_{(i,v)(j,w)}$ with $(i,v)(j,w)\in E$, and we index the rows by the parameters $\beta_{(\ell,x)}$ for each $(\ell,x) \in V$ and $\alpha_{\ell_1,\ell_2}$ for each combination of blocks $\{\ell_1,\ell_2\} \subseteq [k]$ where it may be the case $\ell_1=\ell_2$. The entry in row $\beta_{(\ell,x)}$ and column $p_{(i,v)(j,w)}$ equals 1 if either $(\ell,x) = (i,v)$ or $(\ell,x)=(j,w)$, and equals 0 otherwise, and the entry in row $\alpha_{\ell_1,\ell_2}$ and column $p_{(i,v)(j,w)}$ equals 1 if $\{ \ell_1,\ell_2\} = \{i,j\}$ and equals 0 otherwise. The $\beta$ rows are ordered lexicographically, followed by the $\alpha$ rows, also ordered lexicographically. 

\begin{example}
Consider the model $\mathcal{M}(3,2)$, i.e.,~the model containing blocks $V_1=\{(1,1),(1,2),(1,3)\}$ and $V_2=\{(2,1),(2,2)\}$. This model is specified by $\binom{3+2}{2}=10$ edge probabilities. The design matrix corresponding to this model is provided here.

{\footnotesize
\[
    \NiceMatrixOptions{code-for-first-row = \color{blue},
code-for-first-col = \color{blue}}
\begin{bNiceArray}{CCCCCCCCCC}[first-row,first-col]
& \scriptstyle{p_{(1,1)(1,2)}} & \scriptstyle{p_{(1,1)(1,3)}} & \scriptstyle{p_{(1,1)(2,1)}} & \scriptstyle{p_{(1,1)(2,2)}} & \scriptstyle{p_{(1,2)(1,3)}} & \scriptstyle{p_{(1,2)(2,1)}} & \scriptstyle{p_{(1,2)(2,2)}} & \scriptstyle{p_{(1,3)(2,1)}} & \scriptstyle{p_{(1,3)(2,2)}} & \scriptstyle{p_{(2,1)(2,2)}}\\
\scriptstyle{\beta_{(1,1)}} & 1 & 1 & 1 & 1 & 0 & 0 & 0 & 0 & 0 & 0  \\
\scriptstyle{\beta_{(1,2)}} & 1 & 0 & 0 & 0 & 1 & 1 & 1 & 0 & 0 & 0  \\
\scriptstyle{\beta_{(1,3)}} & 0 & 1 & 0 & 0 & 1 & 0 & 0 & 1 & 1 & 0  \\
\scriptstyle{\beta_{(2,1)}} & 0 & 0 & 1 & 0 & 0 & 1 & 0 & 1 & 0 & 1  \\
\scriptstyle{\beta_{(2,2)}} & 0 & 0 & 0 & 1 & 0 & 0 & 1 & 0 & 1 & 1  \\
\hline
\scriptstyle{\alpha_{1,1}} & 1 & 1 & 0 & 0 & 1 & 0 & 0 & 0 & 0 & 0  \\
\scriptstyle{\alpha_{1,2}} & 0 & 0 & 1 & 1 & 0 & 1 & 1 & 1 & 1 & 0  \\
\scriptstyle{\alpha_{2,2}} & 0 & 0 & 0 & 0 & 0 & 0 & 0 & 0 & 0 & 1  \\
\end{bNiceArray}
\]}
Note that the upper $5 \times 10$ submatrix is the vertex-edge incidence matrix of the complete graph on 5 vertices. This submatrix is the design matrix for the second hypersimplex $\Delta_{2, 5}$.
\end{example}

For a log-linear model given by a design matrix $A$, we denote by $V(A)$ the Zariski closure of the model $\mathcal{M}_A$, and we denote by $I_A:= I(V(A))$ the defining ideal of $V(A)$.  For log-linear models, the ideal $I_A$ is a toric ideal generated by binomials (see \cite{aoki2012markov}, \cite{DSS2009}, \cite{sullivant2018}).  Furthermore, for log-linear models, given an observation $\mathbf{u}$, the maximum likelihood estimate $\hat{\textbf{p}}$ is the unique non-negative solution of the system of polynomial equations given by
\[
A{\textbf{p}}=A\textbf{u} \text{ and } \textbf{p} \in V(A).
\]

For $\beta$-stochastic block models, a generating set for the ideal $I_A$ is given in  \cite{ADP2024}.  In particular, the authors of \cite{ADP2024} show that $I_A$ is generated by quadratics. Each quadratic binomial in $I_A$ for a $\beta$-SBM corresponds to a Markov move, i.e. a move between two graphs with the same sufficient statistic. These moves exchange one nonadjacent pair of edges for another pair of edges on the same four distinct vertices. The quadratic equations obtained by setting each element of the generating set equal to zero, together with the linear equations $A{\textbf{p}}=A\textbf{u}$, constitute the \textit{likelihood equations}, (see \cite{HKS2005}, \cite{DSS2009},\cite{sullivant2018}). These quadratic equations are categorized in \cref{tab:binomials} with names that we will use to refer to both the equations and their corresponding binomials, as needed. \cref{tab:binomials} also shows the graph moves for these equations, using colors to emphasize the distinction between different blocks.

\begin{table}[ht]
\centering
\caption{Quadratic likelihood equations and their corresponding graph moves}
\label{tab:binomials}

\medskip
\begin{tabular}{ccc}
%\hline
\textbf{Name} & \textbf{Equation Structure} & \textbf{Graph Structure}
\\
\hline
\begin{tabular}{c}
     within  \\
    block
\end{tabular}  & {\begin{minipage}{1.8in}\begin{align*}&p_{(i,t)(i,u)}p_{(i,v)(i,w)}& \\ &-p_{(i,t)(i,v)}p_{(i,u)(i,w)}=0\end{align*}\end{minipage}} & \begin{tabular}{c@{\hskip 0.1in}c}
\begin{tikzpicture}[
scale=0.8,
fill opacity=0.7,
single/.style={circle,draw=black,fill=black},
block/.style={oval,draw=blue!50,fill=blue!20,thick},
]
\small

\node[] at (0,1.75) {};

\draw[fill=blue, fill opacity=0.2] (1.5,0) ellipse (1.7cm and 1.7cm){};

\node[single] at (1,0.5) (a){};
\node[] at (0.8,1.0) {$(i,t)$};
\node[single] at (1,-0.5) (b){};
\node[] at (0.8,-1.0) {$(i,u)$};
\node[single] at (2,0.5) (c){};
\node[] at (2.2,1.0) {$(i,v)$};
\node[single] at (2,-0.5) (d){};
\node[] at (2.2,-1.0) {$(i,w)$};

\draw[-] (a) -- (b);
\draw[-] (c) -- (d);

\node[font=\huge] at (4.3,0) {$\Longleftrightarrow$};

\draw[fill=blue, fill opacity=0.2] (7.1,0) ellipse (1.7cm and 1.7cm){};

\node[single] at (6.6,0.5) (a){};
\node[] at (6.4,1.0) {$(i,t)$};
\node[single] at (6.6,-0.5) (b){};
\node[] at (6.4,-1.0) {$(i,u)$};
\node[single] at (7.6,0.5) (c){};
\node[] at (7.8,1.0) {$(i,v)$};
\node[single] at (7.6,-0.5) (d){};
\node[] at (7.8,-1.0) {$(i,w)$};

\draw[-,thick] (a) -- (c);
\draw[-,thick] (b) -- (d);

\end{tikzpicture}
\end{tabular} \\
\\
3-1  & {\begin{minipage}{1.8in}\begin{align*}&p_{(i,t)(j,w)}p_{(i,u)(i,v)}&\\ &-p_{(i,t)(i,u)}p_{(i,v)(j,w)}=0\end{align*}\end{minipage}} & \begin{tabular}{r@{\hskip 0.1in}l}
\begin{tikzpicture}[
scale=0.95,
fill opacity=0.7,
single/.style={circle,draw=black,fill=black},
block/.style={oval,draw=blue!50,fill=blue!20,thick},
]
\small
\node[] at (0,1.3) {};

\draw[fill=blue, fill opacity=0.2] (0,0) ellipse (.8cm and 1.3cm){};
\draw[fill=green, fill opacity=0.2](1.8,0)  ellipse (.8cm and 1.3cm){};

\node[single] at (0.3,0) (a){};
\node[] at (-0.3,0.7) {$(i,t)$};
\node[single] at (0.3,0.7 ) (b){};
\node[] at (-0.3,0.0) {$(i,u)$};
\node[single] at (0.3,-0.7) (c){};
\node[] at (-0.3,-0.7) {$(i,v)$};
\node[single] at (1.5,0) (d){};
\node[] at (2.1,0) {$(j,w)$};

\draw[-,thick] (a) -- (c);
\draw[-,thick] (b) -- (d);

\node[font=\huge] at (3.5,0) {$\Longleftrightarrow$};

\draw[fill=blue, fill opacity=0.2] (5.2,0) ellipse (.8cm and 1.3cm){};
\draw[fill=green, fill opacity=0.2](7,0)  ellipse (.8cm and 1.3cm){};

\node[single] at (5.5,0) (a){};
\node[] at (4.9,0.7) {$(i,t)$};
\node[single] at (5.5,0.7) (b){};
\node[] at (4.9,0.0) {$(i,u)$};
\node[single] at (5.5,-0.7) (c){};
\node[] at (4.9,-0.7) {$(i,v)$};
\node[single] at (6.7,0) (d){};
\node[] at (7.3,0) {$(j,w)$};

\draw[-,thick] (a) -- (b);
\draw[-,thick] (c) -- (d);
\end{tikzpicture}
\end{tabular}
 \\
\\
2-2  & {\begin{minipage}{1.8in}\begin{align*}&p_{(i,t)(j,v)}p_{(i,u)(j,w)}&\\ &-p_{(i,t)(j,w)}p_{(i,u)(j,v)}=0\end{align*}\end{minipage}} & \begin{tabular}{l@{\hskip 0.1in}l}
\begin{tikzpicture}[
scale=0.95,
fill opacity=0.7,
single/.style={circle,draw=black,fill=black},
block/.style={oval,draw=blue!50,fill=blue!20,thick},
]
\small

\node[] at (0,1.3) {};

\draw[fill=blue, fill opacity=0.2](0,0)  ellipse (.8cm and 1.3cm){};
\draw[fill=green, fill opacity=0.2](1.8,0)  ellipse (.8cm and 1.3cm){};

\node[single] at (0.3,0.5) (a){};
\node[] at (-0.3,0.5) {$(i,t)$};
\node[single] at (0.3,-0.5) (b){};
\node[] at (-0.3,-0.5) {$(i,u)$};
\node[single] at (1.5,0.5) (c){};
\node[] at (2.1,0.5) {$(j,v)$};
\node[single] at (1.5,-0.5) (d){};
\node[] at (2.1,-0.5) {$(j,w)$};

\draw[-,thick] (a) -- (c);
\draw[-,thick] (b) -- (d);

\node[font=\huge] at (3.5,0) {$\Longleftrightarrow$};

\draw[fill=blue, fill opacity=0.2](5.2,0)  ellipse (.8cm and 1.3cm){};
\draw[fill=green, fill opacity=0.2](7.0,0)  ellipse (.8cm and 1.3cm){};

\node[single] at (5.5,0.5) (a){};
\node[] at (4.9,0.5) {$(i,t)$};
\node[single] at (5.5,-0.5) (b){};
\node[] at (4.9,-0.5) {$(i,u)$};
\node[single] at (6.7,0.5) (c){};
\node[] at (7.3,0.5) {$(j,v)$};
\node[single] at (6.7,-0.5) (d){};
\node[] at (7.3,-0.5) {$(j,w)$};

\draw[-,thick] (a) -- (d);
\draw[-,thick] (b) -- (c);
\end{tikzpicture}
\end{tabular} \\
\\
2-1-1  & {\begin{minipage}{1.8in}\begin{align*}&p_{(i,t)(j,v)}p_{(i,u)(k,w)}&\\ &-p_{(i,t)(k,w)}p_{(i,u)(j,v)}=0\end{align*}\end{minipage}} & \begin{tabular}{l@{\hskip 0.1in}l}
\begin{tikzpicture}[
fill opacity=0.7,
scale=0.95,
single/.style={circle,draw=black,fill=black},
block/.style={oval,draw=blue!50,fill=blue!20,thick},
]
\small

\node[] at (0,1.4) {};

\draw[fill=blue, fill opacity=0.2](0,0)  ellipse (.8cm and 1.3cm){};
\draw[fill=green, fill opacity=0.2](1.8,0.75)  ellipse (.8cm and 0.6cm){};
\draw[fill=red, fill opacity=0.2](1.8,-0.75)  ellipse (.8cm and 0.6cm){};

\node[single] at (0.3,0.5) (a){};
\node[] at (-0.3,0.5) {$(i,t)$};
\node[single] at (0.3,-0.5) (b){};
\node[] at (-0.3,-0.5) {$(i,u)$};
\node[single] at (1.8,0.5) (c){};
\node[] at (1.8,1.0) {$(j,v)$};
\node[single] at (1.8,-0.5) (d){};
\node[] at (1.8,-1.0) {$(k,w)$};

\draw[-,thick] (a) -- (c);
\draw[-,thick] (b) -- (d);

\node[font=\huge] at (3.5,0) {$\Longleftrightarrow$};
\draw[fill=blue, fill opacity=0.2](5.2,0)  ellipse (.8cm and 1.3cm){};
\draw[fill=green, fill opacity=0.2](7.0,0.75)  ellipse (.8cm and 0.6cm){};
\draw[fill=red, fill opacity=0.2](7.0,-0.75)  ellipse (.8cm and 0.6cm){};

\node[single] at (5.5,0.5) (a){};
\node[] at (4.9,0.5) {$(i,t)$};
\node[single] at (5.5,-0.5) (b){};
\node[] at (4.9,-0.5) {$(i,u)$};
\node[single] at (7,0.5) (c){};
\node[] at (7,1.0) {$(j,v)$};
\node[single] at (7,-0.5) (d){};
\node[] at (7,-1.0) {$(k,w)$};

\draw[-, thick] (a) -- (d);
\draw[-, thick] (b) -- (c);

\end{tikzpicture}
\end{tabular} \\

\end{tabular}

\end{table}

\begin{example}
Consider the model $\mathcal{M}(4,2,1)$. The design matrix has 7 $\beta$-rows and 6 $\alpha$-rows, resulting in 13 linear equations. There are 3 within block quadratic equations, all within block 1, including, e.g. $p_{(1,1)(1,2)}p_{(1,3)(1,4)} -p_{(1,1)(1,3)}p_{(1,2)(1,4)}= 0$. There are 36 equations of the form 3-1, including, e.g. the equation $p_{(1,1)(1,2)}p_{(1,3)(2,1)} - p_{(1,1)(1,3)}p_{(1,2)(2,1)}=0$. There are 6 equations of the form 2-2, including, e.g. the equation $p_{(1,1)(2,1)}p_{(1,2)(2,2)} - p_{(1,1)(2,2)}p_{(1,2)(2,1)}=0$. Finally, there are 16 equations of the form 2-1-1, including, e.g. $p_{(1,1)(2,1)}p_{(2,2)(3,1)} - p_{(1,1)(2,2)}p_{(2,1)(3,1)} = 0$. In total, there are 13 linear and 61 quadratic likelihood equations for this model.
\end{example}

Let $M$ be a $\beta$-SBM $\mathcal{M}(n_1,n_2,\ldots,n_k)$ and let $\textbf{u} = \big( u_{(i,v)(j,w)} \big) \in\mathbb{C}^{|E|}$ be a generic point. Then we will denote by $L(M)$ the system of likelihood equations consisting of the linear equations $A(\textbf{p}-\textbf{u})=\textbf{0}$ and the quadratic equations described in \cref{tab:binomials}. The number of solutions $\textbf{p} = \big( p_{(i,v)(j,w)} \big) \in\mathbb{C}^{|E|}$ to the system $L(M)$ is known as the \textit{maximum likelihood degree} of the model \cite{HKS2005}, \cite{DSS2009}, \cite{sullivant2018} which we denote by $\MLdeg(M)$. 

\begin{lemma}\label{lem:ReorderBlocks}
   Let $k$ be a positive integer and $\tau:[k]\rightarrow[k]$ be a permutation of $[k]$. If $M_1=\mathcal{M}(n_1,n_2,\ldots,n_k)$ and $M_2=\mathcal{M}(n_{\tau(1)},n_{\tau(2)},\ldots,n_{\tau(k)})$ are $\beta$-SBMs then $\MLdeg(M_1)=\MLdeg(M_2)$.
\end{lemma}

\begin{proof}
   The solutions of $L(M_1)$ and $L(M_2)$ differ only by a permutation of coordinates and so the two sets of solutions are equicardinal.
\end{proof}

Throughout the remainder of this manuscript, will use the notation $\MLdeg(n_1,n_2,\ldots,n_k)$ to mean $\MLdeg(M)$, where $M$ is a $\beta$-SBM $\mathcal{M}(n_1,n_2,\ldots,n_k)$. By \cref{lem:ReorderBlocks}, $\MLdeg(n_1,n_2,\ldots,n_k)$ is well-defined.

\section{The Main Theorem}\label{sec:MainThm}

In this section, we state the main theorem, give an example of its application, and state three immediate corollaries. We will prove the main theorem in Section \ref{sec:MainProof}.

\begin{theorem}[The Maximum Likelihood Degree of $\beta$-SBMs] \label{thm:RecursionTheorem}
   Given positive integers $n,k,n_1,n_2,\dots, n_k$ such that $n,k>1$,
   \begin{enumerate}
      \item $\MLdeg(n)=
         \begin{cases}
                    1 & n=2\\
            2^{n-1}-n & n>2
         \end{cases}$,
      \item $\MLdeg(n_1,1)=\MLdeg(n_1+1)$, and
      \item $\MLdeg(n_1,n_2,\dots,n_k)=\displaystyle\prod_{i \in [k]}\MLdeg(n_i,1)$.
   \end{enumerate}
\end{theorem}

\begin{example} Consider the model $\mathcal{M}(5,3,1,6,1,2)$ consisting of six blocks with 5, 3, 1, 6, 1, and 2 vertices respectively. Using \cref{thm:RecursionTheorem}, we can calculate the ML degree of $\mathcal{M}(5,3,1,6,1,2)$ as:
   \begin{align*}
      \MLdeg(5,3,1,6,1,2)&=\MLdeg(5,1)\MLdeg(3,1)\MLdeg(1,1)\MLdeg(6,1)\MLdeg(1,1)\MLdeg(2,1)\\
                     &= \MLdeg(6)\MLdeg(4)\MLdeg(2)\MLdeg(7)\MLdeg(2)\MLdeg(3) \\
                     &=(2^5-6)(2^3-4)(1)(2^6-7)(1)(2^2-3)=5928.
   \end{align*}
\end{example}

\begin{remark}
   Note that $\MLdeg(n) = A(n-1, 1)$ and $\MLdeg(n,1)=A(n,1)$, where $A(n-1, 1)$ and $A(n,1)$ are Eulerian numbers. The Eulerian number $A(n,k)$ is the number of permutations of the numbers 1 to $n$ with $k$ ascents. As discussed in \cite[Remark 23]{ABBGHHNRS2019}, the ML degree of the second hypersimplex $\Delta_{2, n}$, which corresponds to the $\beta$-SBM model $\mathcal M(n)$ is exactly $A(n-1, 1)$.
\end{remark}

By combining the three parts of \cref{thm:RecursionTheorem}, we obtain a formula involving a product of Eulerian numbers for the ML degree of a $\beta$-SBM of arbitrary size.

\begin{corollary}
Let $k,n_1,n_2,\dots,n_k$ be positive integers such that $k>1$. If $N\subseteq [k]$ is the set of indices for blocks containing more than two vertices, 
\[
\MLdeg(n_1,n_2,\dots,n_k)=\prod_{i\in N}\MLdeg(n_i,1)=\prod_{i\in N}(2^{n_i}-(n_i+1)),
\]
when $N$ is non-empty, and $\MLdeg(n_1,\dots,n_k)=1$ otherwise (when all blocks have size 1 or 2).
\end{corollary}

\cref{thm:RecursionTheorem} shows that the ML degree of a $\beta$-SBM is a product of ML degrees of models involving the blocks, thus, an immediate corollary is that augmenting a model with either new blocks or new vertices in existing blocks does not decrease its maximum likelihood degree. In other words, the maximum likelihood degree of $\beta$-stochastic blockmodels exhibits a type of monotonicity.

\begin{corollary}
Let $M_1=\mathcal{M}(m_1,m_2,\dots,m_j)$ and $M_2=\mathcal{M}(n_1,n_2,\dots,n_k)$ be $\beta$-SBMs. If $j\leq k$ and $m_i\leq n_i$ for all $1\leq i\leq j$ then $\MLdeg(M_1)\leq\MLdeg(M_2)$.
\end{corollary}

\section{Proof of the Main Theorem}\label{sec:MainProof}

In this section we provide the proof of the main theorem, \cref{thm:RecursionTheorem}, beginning with part one. 

\begin{proof}[Proof of \cref{thm:RecursionTheorem} Part 1:]\,

Let $n$ be a positive integer such that $n\geq 2$, and let $M$ be the single block $\beta$-SBM on $n$ vertices with design matrix $A$. If $n=2$, the system of equations $L(M)$ consists only of the linear system $A(\textbf{p}-\textbf{u})=\textbf{0}$, and since any $G \in \mathcal G_2$ has at most one edge, the matrix $A$ has only a single column. Thus, $\ker(A)=\textbf{0}$ and $L(M)$ has a unique solution $\textbf{p}=\textbf{u}$, so $\MLdeg(M)=1$.

Now suppose $n>2$. The single $\alpha$ row $\alpha_{1,1}$, which consists of all 1's, in the design matrix of $A$ is a linear combination of the $\beta$ rows. Thus, $A$ and the matrix formed by removing row $\alpha_{1,1}$ are row equivalent and this new matrix is the incidence matrix of the complete graph on $n$ vertices. By \cite[Remark 23]{ABBGHHNRS2019}, the maximum likelihood degree is $2^{n-1}-n$, as the convex hull of the columns of this matrix is an $n-1$ dimensional polytope known as the second hypersimplex of order $n$. 
\end{proof}

Now we consider $\beta$-SBMs with more than one block. In the case where we have two blocks, one with a single vertex, the ML degree is the same as the ML degree for the single block model with the same number of vertices. 

\begin{proof}[Proof of \cref{thm:RecursionTheorem} Part 2:]\

Let $n_1$ be a positive integer. Let  $M$ and $M'$, be $\beta$-SBM's, with $n_1+1$ vertices. Let $M$ have a single block with vertex set $V_1=\{(1,v):1\leq v\leq n_1+1\}$ and let $M'$ have two blocks with vertex sets $V_1'=\{(1,v):1\leq v\leq n_1\}$ and $V_2'=\{(2,1)\}$. The quadratic likelihood equations for $M$ that do not use vertex $(1,n_1+1)$ are identical to those for $M'$ that do not use vertex $(2,1)$. Moreover, there is a one-to-one correspondence between the within block quadratics of $M$ that use vertex $(1,n_1+1)$ and the 3-1 quadratics of $M'$ that use vertex $(2,1)$. These are the only quadratic likelihood equations for these models, and their solution sets are the same after relabeling the singleton vertex. 

Now consider the linear equations. Let $A$ and $A'$ be the design matrices for $M$ and $M'$ respectively. First, notice that the $\beta$ rows of $A$ and $A'$ are the same if we identify the $(1,n_1+1)$ vertex of $M$ with the $(2,1)$ vertex of $M'$. Since $M$ has only one block, $A$ has only one $\alpha$ row, $\alpha_{1,1}$, consisting of all 1's. The $A'$ design matrix has three $\alpha$ rows, and we can go from $A$ to $A'$ using row operations as follows. The $\alpha_{1,1}$ row of $A'$ is obtained by subtracting the $\beta_{(1,n_1+1)}$ row of $A$ (which is the same as the $\beta_{(2,1)}$ row of $A'$) from the $\alpha_{1,1}$ row of $A$. The $\alpha_{1,2}$ row of $A'$ is the same as the $\beta_{(1,n_1+1)}$ row of $A$, and the $\alpha_{2,2}$ row of $A'$ consists of all 0's. Thus, $A(\mathbf{p}-\mathbf{u})=0$ and $A'(\mathbf{p}-\mathbf{u})=0$ have the same solution set if we relabel the $(2,1)$ vertex of $M'$, and therefore $\MLdeg(n_1+1)=\MLdeg(n_1,1)$. 

\end{proof}

We will need the following lemma  to prove \cref{thm:RecursionTheorem} Part 3; the proof of the lemma is in \cref{sec:factoringproof}.

\begin{lemma}[Factoring Lemma]\label{lem:FactoringLemma}
Let $M$ be a $\beta$-SBM $\mathcal{M}(n_1,n_2,\ldots,n_k)$ with $k>1$. Then
\[\MLdeg(n_1,n_2,\dots,n_k)=\MLdeg(n_1,n_2,\dots,n_{k-1},1)\MLdeg(n_k,1).\]
\end{lemma}

\begin{proof}[Proof of \cref{thm:RecursionTheorem} Part 3:]\,

Let $k,n_1,n_2,\dots,n_k$ be positive integers such that $k>1$. Given a $\beta$-SBM $\mathcal{M}(n_1,n_2,\dots,n_k)$ with design matrix $A$, define $m=|\{i\in[k]:n_i>1\}|$, i.e., the number of blocks having more than one vertex, and proceed by induction on $m$. For the base case, assume $m=0$, that is $n_1=n_2=\cdots=n_k=1$. Then there are no quadratic equations in $L(M)$, only linear equations given by $A\mathbf{p}=A\mathbf{u}$. Let $e\in E$ be a potential edge for $M$. Then $e=(i,1)(j,1)$ for $i,j\in [k]$ such that $i\neq j$. Notice that the $\alpha_{i,j}$ row of $A$ consists of a single 1 in the $p_{(i,1)(j,1)}$ column, so $p_e=u_e$. Thus, $A\mathbf{p}=A\mathbf{u}$ has only the solution $\mathbf{p}=\mathbf{u}$ and since $\MLdeg(1,1)=1$ by \cref{thm:RecursionTheorem} Parts 1 and 2,
\[
\MLdeg(n_1,n_2,\dots,n_k)=1=\prod_{i\in[k]} 1=\prod_{i\in[k]}\MLdeg(1,1)=\prod_{i\in[k]}\MLdeg(n_i,1),
\]
proving the base case. 

Now suppose that the claim is true for some fixed but arbitrary $m=m_0\geq 0$ and consider a model $M=\mathcal{M}(n_1,n_2,\dots,n_k)$ with $m=m_0+1\geq 1$ blocks having more than one vertex. \cref{lem:ReorderBlocks} allows us to assume that $n_k>1$. Then by \cref{lem:FactoringLemma},
$\MLdeg(n_1,n_2,\dots,n_k)=\MLdeg(n_1,n_2,\dots,n_{k-1},1)\MLdeg(n_k,1).$
The model in the first factor now has $m_0$ blocks with more than one vertex, so we can apply the inductive hypothesis and the fact that $\MLdeg(1,1)=1$ to get 
\[
\MLdeg(n_1,n_2,\dots,n_{k-1},1)=\left ( \prod_{i=1}^{k-1}\MLdeg(n_i,1) \right ) \MLdeg(1,1) = \prod_{i=1}^{k-1}\MLdeg(n_i,1),
\]
and therefore
$\MLdeg(n_1,n_2,\dots,n_k)=\prod_{i=1}^{k}\MLdeg(n_i,1),$ as desired.
\end{proof}

\section{Proof of the Factoring Lemma}

The proof of the main theorem relies heavily on the Lemma \ref{lem:FactoringLemma}, the Factoring Lemma, and the proof of Lemma \ref{lem:FactoringLemma} is where the main work of this paper lies.  The proof hinges on the fact that there is a complete description of the generating set of the ideal $I_A$. 

Throughout this section, we will use the following notation. Let $k, n_1, n_2, \dots, n_k$ be positive integers with $k>1$, and let $M=\mathcal{M}(n_1,n_2,\dots,n_k)$ be a $\beta$-SBM where block $i$ has vertex set $V_i=\{(i,v):1\leq v\leq n_i\}$ for each $i\in [k]$, and potential edge set $E=\bigcup_{i=1}^k\bigcup_{j=1}^k E_{i,j}$ as previously defined. Let $A$ be the design matrix for $M$. Given a generic point $\textbf{u}\in\mathbb{C}^{|E|}$, we denote the system of likelihood equations for $M$ by $L(M)$, with $\mathcal{S}\subseteq\mathbb{C}^{|E|}$ being the corresponding solution set.

\subsection{Constructing new models by contracting blocks}

To prove \cref{lem:FactoringLemma}, we  introduce new models, $M_1$, $M_2$,  each arising from the model $M$ by collapsing a particular portion of $M$ to a single vertex block. Call the vertex set of this block $V_{*}=\{(*,1)\}$.

\begin{definition}\label{def:M1CollapseLast}
Let $M = \mathcal M(n_1, \ldots, n_k)$ be a $\beta$-SBM.  Let $M_1$ be a modification of the model $M$ obtained by the contraction of block $V_k$ to a single vertex $(*,1)$. The set of potential edges for $M_1$ is
\[
E_1=\left (\bigcup_{i=1}^{k-1}\bigcup_{j=i}^{k-1} E_{i,j}\right ) \cup \left (\bigcup_{i=1}^{k-1} E_{i,*} \right )
\quad \quad 
\text{where} \quad \quad 
E_{i,*} = \{(i,v)(*,1): 1\leq v\leq n_i\}.
\]
We will denote the design matrix for $M_1$ as $A_1$. Suppose $\mathbf{p}\in\mathcal{S}$, and let $\mathbf{u}\in \mathbb{C}^{|E|}$ be a generic point. For each $i\in[k-1]$ and $v\in[n_i]$, define $p_{(i,v)(*,1)}:=\sum_{w=1}^{n_{k}} p_{(i,v)(k,w)}$ and $u_{(i,v)(*,1)}:=\sum_{w=1}^{n_{k}} u_{(i,v)(k,w)}$. Define $\textbf{u}_1 \in \mathbb{C}^{|E_1|}$ so that the coordinate indexed by $e\in E_1$ is $u_e$.  In particular, the coordinates of $\mathbf{u}$ and $\mathbf{u_1}$ agree on all potential edges that appear in both $M$ and $M_1$, while the remaining coordinates, corresponding to the edges in $E_{i,*}$ are defined above. Finally, we will refer to the set of complex solutions to the system $L(M_1)$ as $\mathcal{S}_1$. Note that
$|\mathcal{S}_1|=\MLdeg(M_1)=\MLdeg(n_1,n_2,\dots, n_{k-1},1).$
\end{definition}

\begin{definition}\label{def:M2CollapseFirst}
Let $M = \mathcal M(n_1, \ldots, n_k)$ be a $\beta$-SBM.  Let $M_2$ be a modification of the model $M$ obtained by the contraction of blocks $V_1, V_2,\dots, V_{k-1}$ to a single vertex $(*,1)$. The set of potential edges for $M_2$ is
\[
E_2=E_{*,k} \cup E_{k,k}
\quad \quad 
\text{where}\quad \quad 
E_{*,k}=\{(*,1)(k,w): 1\leq w\leq n_k\}.
\]
We will denote the design matrix for $M_2$ as $A_2$. Suppose $\mathbf{p}\in\mathcal{S}$, and let $\mathbf{u}\in \mathbb{C}^{|E|}$ be a generic point. For each $w\in[n_k]$, define $p_{(*, 1)(k,w)}:=\sum_{i=1}^{k-1}\left(\sum_{v=1}^{n_i} p_{(i,v)(k,w)}\right)$ and $u_{(*, 1)(k,w)}:=\sum_{i=1}^{k-1}\left(\sum_{v=1}^{n_i} u_{(i,v)(k,w)}\right)$. Define $\textbf{u}_2 \in \mathbb{C}^{|E_2|}$ so that the coordinate indexed by $e\in E_2$ is $u_e$. In particular, the coordinates of $\mathbf{u}$ and $\mathbf{u_2}$ agree on all potential edges that appear in both $M$ and $M_2$, while the remaining coordinates, corresponding to the edges in $E_{*,k}$ are defined above.  Finally, we will refer to the set of complex solutions to the system $L(M_2)$ as $\mathcal{S}_2$. Note that
$|\mathcal{S}_2|=\MLdeg(M_2)=\MLdeg(n_k,1).$
\end{definition}

\begin{figure}[ht]
\centering
\begin{tikzpicture}
   \draw (7.5,4.25) node{$M_1$};
   \draw [decorate, decoration = {brace,amplitude=10pt}] (2,3.25) --  (12.75,3.25);
   \foreach \x/\i in {3/1,5.5/2,9.375/{k-1},14.25/k}{%
      \draw (\x,.8) circle (1.2 cm and 2cm);
      \foreach \y/\v in {1,2,3.65/{n_{\i}}}{%
         \fill (\x,{3.5-\y}) circle (0.1cm);
         \draw (\x,{3.2-\y}) node{\footnotesize$(\i,\v)$};
      }%
      \draw (\x,.6) node{$\vdots$};
   }%
   \draw (7.5,0.8) node{\Large$\cdots$};
   \fill (11.8,.8) circle (0.1cm);
   \draw (11.8,0.4) node{\footnotesize$(*,1)$};
   \draw [decorate, decoration = {brace,mirror,amplitude=10pt}] (11.5,-1.5) --  (15.5,-1.5);
   \draw (13.5,-2.25) node{$M_2$};
\end{tikzpicture}
\caption{The blocks of $M_1$ and $M_2$. This figure exhibits the blocks of $M$ together with the newly defined block $V_{*}$ while also indicating which blocks are members of $M_1$ and which blocks are members of $M_2$.}
\label{fig:M1M2}
\end{figure}
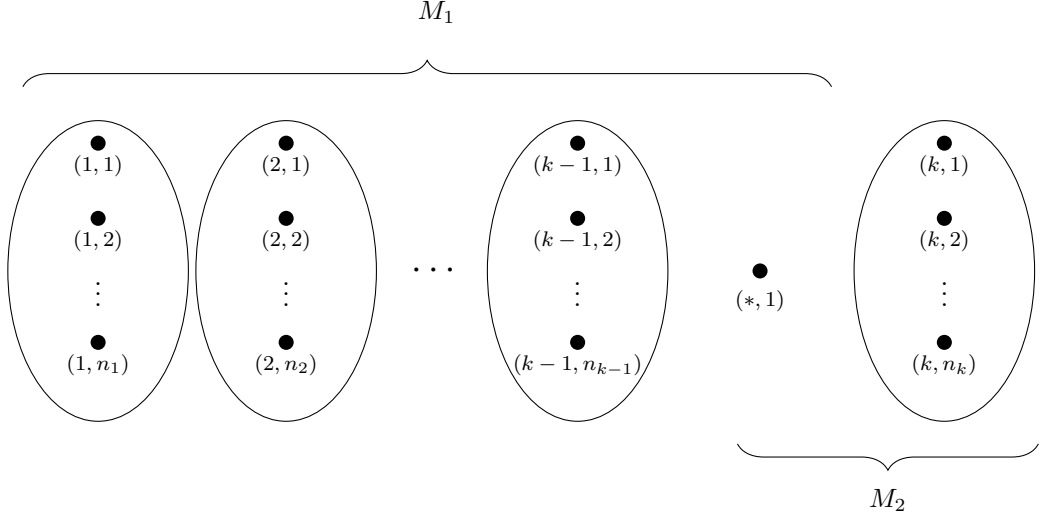

Because our graphs contain no self-adjacent vertices, we introduce the following convention to simplify the index notation for sums. For all $i\in [k] \cup \{*\}$ and for all $(i,v) \in V_i$, define $p_{(i,v)(i,v)}=u_{(i,v)(i,v)}=0$.

\subsection{Proof of Factoring Lemma}\label{sec:factoringproof}

\cref{lem:FactoringLemma} allows us to write the ML degree of model $M$ as the product of the ML degrees of models $M_1$ and $M_2$. To prove this lemma, we will construct a bijective map $\phi:\mathcal{S}\rightarrow \mathcal{S}_1\times\mathcal{S}_2$. To begin, we show that given a solution $\mathbf{p}\in\mathcal{S}$ of  $L(M)$, we can obtain solutions $(\mathbf{p}_1,\mathbf{p}_2)\in\mathcal{S}_1\times\mathcal{S}_2$.

\begin{lemma}\label{lem:p1p2Solutions}
Let $\textup{\textbf{u}}\in\mathbb{C}^{|E|}$ be  generic, and let $\textup{\textbf{p}}\in\mathcal{S}$ be a particular solution of the system $L(M)$, both indexed by $E$ with coordinates $u_{(i,v)(j,w)}$ and $p_{(i,v)(j,w)}$ for $(i,v)(j,w)\in E$, respectively. Let $\textup{\textbf{u}}_1$, $\mathcal{S}_1$, $\textup{\textbf{u}}_2$, and $\mathcal{S}_2$ be as defined in Definitions \ref{def:M1CollapseLast} and \ref{def:M2CollapseFirst}. If $\textbf{p}_1 \in \mathbb{C}^{|E_1|}$ is defined so that the coordinate indexed by $e\in E_1$ is $p_e$ and $\textbf{p}_2 \in\mathbb{C}^{|E_2|}$ is defined so that the coordinate indexed by $e\in E_2$ is $p_e$, then $(\textup{\textbf{p}}_1,\textup{\textbf{p}}_2)\in \mathcal{S}_1\times \mathcal{S}_2$.
\end{lemma}

\begin{proof}
 Let $\mathbf{u}\in\mathbb{C}^{|E|}$ and define $\mathbf{u}_1$ and $\mathbf{u}_2$ as above. Assume $\mathbf{p}\in\mathcal{S}$ and define $\mathbf{p}_1$ and $\mathbf{p}_2$ as in the statement of the lemma. Since $\mathbf{p}\in\mathcal{S}$, $A\mathbf{p}=A\mathbf{u}$, and $\mathbf{p}$ satisfies the quadratic equations in $L(M)$. 
First we show that $A_1\mathbf{p}_1=A_1\mathbf{u}_1$. Let $i\in [k-1]$ and $v\in [n_i]$, and consider the $\beta_{(i,v)}$ row of $A$. Then since $A\mathbf{p}=A\mathbf{u}$, 
\[
\left(\sum_{j=1}^{k-1}\left(\sum_{w=1}^{n_j} p_{(i,v)(j,w)}\right)\right)+\sum_{w=1}^{n_k} p_{(i,v)(k,w)}=\left(\sum_{j=1}^{k-1}\left(\sum_{w=1}^{n_j} u_{(i,v)(j,w)}\right)\right)+\sum_{w=1}^{n_k} u_{(i,v)(k,w)}
\]
and by definition of $p_{(i,v)(*,1)}$ and $u_{(i,v)(*,1)}$ (\cref{def:M1CollapseLast}), 
\[
\left(\sum_{j=1}^{k-1}\left(\sum_{w=1}^{n_j} p_{(i,v)(j,w)}\right)\right)+p_{(i,v)(*,1)}=\left(\sum_{j=1}^{k-1}\left(\sum_{w=1}^{n_j} u_{(i,v)(j,w)}\right)\right)+u_{(i,v)(*,1)}
\]
which is the $\beta_{(i,v)}$ row of $A_1\mathbf{p}_1=A_1\mathbf{u}_1$. Now let $i\in[k-1]$ and consider the $\alpha_{i,k}$ row of $A\mathbf{p}=A\mathbf{u}$:
\begin{equation}\label{eqn:Alphaikrow1}
\sum_{v=1}^{n_i}\left(\sum_{w=1}^{n_k} p_{(i,v)(k,w)}\right)=\sum_{v=1}^{n_i}\left(\sum_{w=1}^{n_k} u_{(i,v)(k,w)}\right).
\end{equation}
By definition of $p_{(i,v)(*,1)}$ and $u_{(i,v)(*,1)}$, we get
\begin{equation}\label{eqn:Alphaikrow2}
\sum_{v=1}^{n_i}p_{(i,v)(*,1)}=\sum_{v=1}^{n_i}u_{(i,v)(*, 1)}. 
\end{equation}
Since this equality holds for each $i\in[k-1]$, we can sum over $i$, which gives
$\sum_{i=1}^{k-1}\left(\sum_{v=1}^{n_i}p_{(i,v)(*,1)}\right)=\sum_{i=1}^{k-1}\left(\sum_{v=1}^{n_i}u_{(i,v)(*, 1)}\right)$, the $\beta_{(*,1)}$ row of $A_1\mathbf{p}_1=A_1\mathbf{u}_1$. Notice that the $\alpha_{i,k}$ row above, in \cref{eqn:Alphaikrow2} is also the $\alpha_{i,*}$ row of $A_1\mathbf{p}_1=A_1\mathbf{u}_1$. For $i,j\in[k-1]$, the $\alpha_{i,j}$ row in $A_1\mathbf{p}_1=A_1\mathbf{u}_1$ provides the same equation as the corresponding $\alpha_{i,j}$ row in $A\mathbf{p}=A\mathbf{u}$. Thus, we have checked all rows and $A_1\mathbf{p}_1=A_1\mathbf{u}_1$.

Now we show that $\mathbf{p}_1$ satisfies the quadratic equations in $L(M_1)$. First note that all quadratic equations in $L(M_1)$ that are within and between the first $k-1$ blocks, $V_1,V_2,\dots,V_{k-1}$, are the same as the quadratic equations within and between blocks $V_1,V_2,\dots,V_{k-1}$ in $L(M)$. Thus $\mathbf{p}_1$ satisfies all quadratic equations in $L(M_1)$ that are between and within the first $k-1$ blocks of $L(M_1)$. This leaves us only with the 3-1 and 2-1-1 quadratic equations involving block $V_*$.

Let $i,j\in[k-1]$,  $a,b\in[n_i]$, and $c\in[n_j]$. Using \cref{def:M1CollapseLast}, we obtain 
\begin{equation}\label{eqn:M1Binomials}
\begin{split}
p_{(i,a)(j,c)}p_{(i,b)(*,1)}-p_{(i,a)(*,1)}p_{(i,b)(j,c)} & = p_{(i,a)(j,c)}\left(\sum_{w=1}^{n_k} p_{(i,b)(k,w)}\right)-\left(\sum_{w=1}^{n_k} p_{(i,a)(k,w)}\right)p_{(i,b)(j,c)} \\
& = \sum_{w=1}^{n_k} \left(p_{(i,a)(j,c)}p_{(i,b)(k,w)}-p_{(i,a)(k,w)}p_{(i,b)(j,c)}\right).
\end{split}
\end{equation}
When $i=j$ and $a,b$, and $c$ are distinct, the left hand side of \cref{eqn:M1Binomials} is a 3-1 binomial involving block $V_*$ in $L(M_1)$, and it is equal to 0 since the right hand side is a sum of 3-1 binomials in $L(M)$. Similarly, when $i\neq j$ and $a\neq b$, the left hand side of \cref{eqn:M1Binomials} is a 2-1-1 binomial involving $V_*$ in $L(M_1)$, and it is equal to 0 since the right hand side is a sum of 2-1-1 binomials in $L(M)$. Thus, we have shown that $\mathbf{p_1}$ satisfies all of the quadratic equations in $L(M_1)$, and therefore, $\mathbf{p_1}\in\mathcal{S}_1$.

Now we show that $A_2\mathbf{p}_2=A_2\mathbf{u}_2$. Let $w\in[n_k]$ and consider the $\beta_{(k,w)}$ row of $A$. Since $A\mathbf{p}=A\mathbf{u}$,
\[
\left(\sum_{i=1}^{k-1}\left(\sum_{v=1}^{n_i} p_{(i,v)(k,w)}\right)\right)+\sum_{c=1}^{n_k} p_{(k,c)(k,w)}=\left(\sum_{i=1}^{k-1}\left(\sum_{v=1}^{n_i} u_{(i,v)(k,w)}\right)\right)+\sum_{c=1}^{n_k} u_{(k,c)(k,w)},
\]
and by definition of $p_{(*,1)(k,w)}$ and $u_{(*,1)(k,w)}$, 
$p_{(*,1)(k,w)}+\sum_{c=1}^{n_k} p_{(k,c)(k,w)}=u_{(*,1)(k,w)}+\sum_{c=1}^{n_k} u_{(k,c)(k,w)}$, which is the $\beta_{(k,w)}$ row of $A_2\mathbf{p}_2=A_2\mathbf{u}_2$. Let $i\in[k-1]$ and consider the $\alpha_{i,k}$ row of $A\mathbf{p}=A\mathbf{u}$, which is \cref{eqn:Alphaikrow1} above. Since we have one of these equations for each $i\in[k-1]$, we can sum over $i$ to obtain
\[
\sum_{i=1}^{k-1}\left(\sum_{v=1}^{n_i}\left(\sum_{w=1}^{n_k} p_{(i,v)(k,w)}\right)\right)=\sum_{i=1}^{k-1}\left(\sum_{v=1}^{n_i}\left(\sum_{w=1}^{n_k} u_{(i,v)(k,w)}\right)\right),
\]
and by \cref{def:M2CollapseFirst}, we have $\sum_{w=1}^{n_k} p_{(*,1)(k,w)}=\sum_{w=1}^{n_k} u_{(*,1)(k,w)}$, which is the $\beta_{(*,1)}$ row of $A_2\mathbf{p}_2=A_2\mathbf{u}_2$. The $\alpha_{k,k}$ row of $A_2\mathbf{p}_2=A_2\mathbf{u}_2$ is the same as the $\alpha_{k,k}$ row of $A\mathbf{p}=A\mathbf{u}$, so we just need the $\alpha_{*, k}$ row of $A_2\mathbf{p}_2=A_2\mathbf{u}_2$, which is the same as the $\beta_{(*,1)}$ row, obtained above. Thus, $A_2\mathbf{p}_2=A_2\mathbf{u}_2$. 

Now we show that $\mathbf{p}_2$ satisfies the quadratic equations in $L(M_2)$. All quadratic equations within block $V_k$ in $L(M_2)$ are the same as the within block $V_k$ quadratic equations from $L(M)$. All we need to check are the 3-1 quadratic equations between block $V_*$ and block $V_k$. Let $a,b,c\in [n_k]$ be distinct. Note that if distinct $a$, $b$ and $c$ do not exist, then $L(M_2)$ has no 3-1 quadratic equations involving block $V_*$. Using \cref{def:M2CollapseFirst}, the 3-1 quadratic binomials in $L(M_2)$ involving $V_*$ satisfy
\begin{eqnarray*}
p_{(k,a)(k,b)}p_{(*,1)(k,c)}&-&p_{(k,b)(k,c)}p_{(*,1)(k,a)}\\
& = & p_{(k,a)(k,b)}\left(\sum_{i=1}^{k-1}\left(\sum_{v=1}^{n_i} p_{(i,v)(k,c)}\right)\right)-p_{(k,b)(k,c)}\left(\sum_{i=1}^{k-1}\left(\sum_{v=1}^{n_i} p_{(i,v)(k,a)} \right)\right) \\
& = & \sum_{i=1}^{k-1}\left(\sum_{v=1}^{n_i}\left(p_{(k,a)(k,b)}p_{(i,v)(k,c)}-p_{(k,b)(k,c)}p_{(i,v)(k,a)}\right)\right) 
 = 0,
\end{eqnarray*}
 where the last equality arises because we have a sum of 3-1 binomials from $L(M)$. Thus, we've shown that $\mathbf{p}_2$ satisfies all of the quadratic equations in $L(M_2)$, and therefore, $\mathbf{p}_2\in \mathcal{S}$. Hence, $(\mathbf{p}_1,\mathbf{p}_2)\in\mathcal{S}_1\times\mathcal{S}_2$. 
\end{proof}

\begin{definition}\label{def:phidef}
Let $\mathbf{u}\in\mathbb{C}^{|E|}$ be generic, and define $\mathcal{S}$, $\mathcal{S}_1$, $\mathcal{S}_2$, $\mathbf{u}_1$, and $\mathbf{u}_2$ as in Definitions \ref{def:M1CollapseLast} and \ref{def:M2CollapseFirst}. Define the function $\phi:\mathcal{S}\rightarrow (\mathcal{S}_1\times \mathcal{S}_2)$ by $\phi(\mathbf{p})=(\mathbf{p}_1,\mathbf{p}_2)$ for $\mathbf{p}\in \mathcal{S}\subseteq \mathbb{C}^{|E|}$ where for $\mathbf{p}_1\in\mathcal{S}_1\subseteq \mathbb{C}^{|E_1|}$, the coordinate indexed by $e\in E_1$ is $p_e$, and for $\mathbf{p}_2\in\mathcal{S}_2\subseteq\mathbb{C}^{|E_2|}$, the coordinate indexed by $e\in E_2$ is $p_e$, as in \cref{lem:p1p2Solutions}.
\end{definition} 

Our goal now is to show $\phi$ is a bijection. First, we show that given solutions to $L(M_1)$ and $L(M_2)$, one can construct a solution to $L(M)$. In particular, we can construct an element of $\mathcal{S}$ from an element of $\mathcal{S}_1$ and an element of $\mathcal{S}_2$. To do this, we will need the following lemma, which relates coordinates of an element of $\mathcal{S}_1$  to those of an element of $\mathcal{S}_2$.

\begin{remark}
Notice that in \cref{lem:AnnoyingLittleEquation} we were interested in an arbitrary element of $\mathcal{S}_1\times\mathcal{S}_2$ rather than one obtained from an element of $\mathcal{S}$ through summation. To help avoid confusion, in \cref{lem:AnnoyingLittleEquation} and in what follows, we use $(\hat{\mathbf{p}}_1,\hat{\mathbf{p}}_2)$ with coordinates $\hat{p}_e$ to denote an arbitrary element of $\mathcal{S}_1\times\mathcal{S}_2$ and $(\mathbf{p}_1,\mathbf{p}_2)$ with coordinates $p_e$ to denote an element of $\mathcal{S}_1\times\mathcal{S}_2$ obtained via the map $\phi$ defined above. 
\end{remark}

\begin{lemma}\label{lem:AnnoyingLittleEquation}
Let $\mathbf{u}\in \mathbb{C}^{|E|}$ and define $\mathbf{u}_1\in \mathbb{C}^{|E_1|}$ and $\mathbf{u}_2\in\mathbb{C}^{|E_2|}$ as in Definitions \ref{def:M1CollapseLast} and \ref{def:M2CollapseFirst}. Let $(\hat{\mathbf{p}}_1, \hat{\mathbf{p}}_2)\in \mathcal{S}_1\times \mathcal{S}_2$ with coordinates, $\hat{p}_e$, of $\hat{\mathbf{p}}_1$ indexed by $e\in E_1$ and coordinates, $\hat{p}_e$, of $\hat{\mathbf{p}}_2$ indexed by $e\in E_2$. Then
$$\sum_{a=1}^{n_k} \hat{p}_{(*,1)(k,a)}=\sum_{j=1}^{k-1}\sum_{b=1}^{n_j} \hat{p}_{(j,b)(*,1)}.$$

\end{lemma}
\begin{proof}
Let $(\hat{\mathbf{p}}_1,\hat{\mathbf{p}}_2)\in\mathcal{S}_1\times\mathcal{S}_2$ as in the statement of the lemma. Then $\hat{\mathbf{p}}_1$ must satisfy all equations in $L(M_1)$, and must in particular satisfy the linear equations $A_1\hat{\mathbf{p}}_1=A_1\mathbf{u}_1$ and $\hat{\mathbf{p}}_2$ must satisfy all equations in $L(M_2)$, and must in particular satisfy the linear equations $A_2\hat{\mathbf{p}}_2=A_2\mathbf{u}_2$. By using this along with the definitions of $\mathbf{u}_{(*,1)(k,a)}$ and $\mathbf{u}_{(j,b)(*,1)}$, we have that
\begin{align*}
        \sum_{a=1}^{n_k} \hat{p}_{(*,1)(k,a)} 
     = \sum_{a=1}^{n_k} u_{(*,1)(k,a)}
     &= \sum_{a=1}^{n_k} \left( \sum_{j=1}^{k-1}\sum_{b=1}^{n_j} u_{(j,b)(k,a)} \right)\\  
     &=\sum_{j=1}^{k-1}\sum_{b=1}^{n_j} \left( \sum_{a=1}^{n_k} u_{(j,b)(k,a)}  \right) 
     =\sum_{j=1}^{k-1}\sum_{b=1}^{n_j}  u_{(j,b)(*,1)}   
     = \sum_{j=1}^{k-1}\sum_{b=1}^{n_j} \hat{p}_{(j,b)(*,1)},
\end{align*}
where the first equality is true by the $\beta_{(*,1)}$ row of $A_2\hat{\mathbf{p}}_2=A_2\mathbf{u}_2$ and the last equality is true by the $\beta_{(*,1)}$ row of $A_1\hat{\mathbf{p}}_1=A_1\mathbf{u}_1$.
\end{proof}

We are now ready to construct an element of $\mathcal{S}$ from an element of $\mathcal{S}_1$ and an element of $\mathcal{S}_2$.

\begin{lemma}\label{lem:Backwards}
Let $\mathbf{u}\in \mathbb{C}^{|E|}$ be generic and define $\mathbf{u}_1\in \mathbb{C}^{|E_1|}$ and $\mathbf{u}_2\in\mathbb{C}^{|E_2|}$ as in Definitions \ref{def:M1CollapseLast} and \ref{def:M2CollapseFirst}. Let $(\hat{\mathbf{p}}_1, \hat{\mathbf{p}}_2)\in \mathcal{S}_1\times \mathcal{S}_2$ with coordinates, $\hat{p}_e$, of $\hat{\mathbf{p}}_1$ indexed by $e\in E_1$ and coordinates, $\hat{p}_e$, of $\hat{\mathbf{p}}_2$ indexed by $e\in E_2$. Let $\mathbf{p} \in \mathbb{C}^{|E|}$ with coordinates $p_e$ indexed in $E$ such that 

\begin{itemize}
    \item[(a)] $p_e=\hat{p}_{e}$ if $e\in \left(E_1\setminus \left(\bigcup_{i=1}^{k-1} E_{i*}\right)\right)\cup (E_2\setminus E_{* k})$ (these coordinates correspond to edges in $M$ within and between blocks $V_1,V_2,\dots, V_{k-1}$ or edges within block $V_k$),

    \item[(b)] and for $i\in [k-1]$, $v\in [n_i]$, and $w\in [n_k]$,
    \begin{equation*}\label{eqn:pijkm}
    p_{(i,v)(k,w)}:= \frac{\hat{p}_{(i,v)(*,1)}\hat{p}_{(*,1)(k,w)}}{\displaystyle\sum_{x=1}^{n_k} \hat{p}_{(*,1)(k,x)}}
    \end{equation*}
    (these coordinates correspond to edges in $M$ between block $V_k$ and another block). 
\end{itemize}
Then $\mathbf{p}\in\mathcal{S}$.
\end{lemma}

\begin{proof}
Let $\mathbf{u}\in\mathbb{C}^{|E|}$ be generic, define $\mathbf{u}_1$ and $\mathbf{u}_2$ as in \cref{def:M1CollapseLast} and \cref{def:M2CollapseFirst}, let
$(\hat{\mathbf{p}}_1,\hat{\mathbf{p}}_2)\in\mathcal{S}_1\times\mathcal{S}_2$, and define $\mathbf{p}$ as in the statement of the lemma. For ease of notation, we let $P:=\left(\sum_{a=1}^{n_k} \hat{p}_{(*,1)(k,a)}\right)^{-1}$, allowing us to rewrite the definition in part (b) of the lemma as $p_{(i,v)(k,w)}:=\hat{p}_{(i,v)(*,1)}\hat{p}_{(*,1)(k,w)}P$ for $i\in[k-1]$, $v\in [n_i]$, $w\in[n_k]$. Note that since $\mathbf{u}$ is generic, we can assume $\sum_{a=1}^{n_k} \hat{p}_{(*,1)(k,a)} \neq 0$. We first  show  that $\mathbf{p}$ satisfies all quadratic equations in $L(M)$.

The quadratic equations in $L(M)$ within and between blocks $V_1,\dots, V_{k-1}$ are the same as the quadratic equations in $L(M_1)$ within and between blocks $V_1,\dots, V_{k-1}$, since $p_{(i,v)(j,u)} = \hat{p}_{(i,v)(j,u)}$ whenever $i,j\in[k-1]$ as defined in part (a) of the lemma. Similarly, the quadratic equations in $L(M)$ within block $V_k$ are the same as the quadratic equations in $L(M_2)$ within block $V_k$, since $p_{(k,u)(k,w)} = \hat{p}_{(k,u)(k,w)}$ as defined in part (a) of the lemma.

This leaves us to consider the quadratic equations in $L(M)$ that are between blocks $V_1,\dots,V_{k-1}$ and block $V_k$. We show that the coordinates of $\mathbf{p}$ satisfy all of these quadratic equations using the definition of $\mathbf{p}$ along with the fact that $\hat{\mathbf{p}}_1$ satisfies all quadratic equations in $L(M_1)$ and $\hat{\mathbf{p}}_2$ satisfies all quadratic equations in $L(M_2)$. 

We first consider the 3-1 quadratic equations in $L(M)$ with 3 vertices in $V_k$. Note that if $n_k<3$, there's nothing to prove. Let $i\in[k-1], v\in[n_i]$, and $a,b,c\in[n_k]$ distinct. Then 
\begin{eqnarray*}
p_{(i,v)(k,a)}p_{(k,b)(k,c)}&-&p_{(i,v)(k,b)}p_{(k,a)(k,c)}\\
& = & \left(\hat{p}_{(i,v)(*,1)}\hat{p}_{(*,1)(k,a)}P\right)\hat{p}_{(k,b)(k,c)}-\left(\hat{p}_{(i,v)(*,1)}\hat{p}_{(*,1)(k,b)}P\right)\hat{p}_{(k,a)(k,c)} \\
& = & \hat{p}_{(i,v)(*,1)}P\left(\hat{p}_{(*,1)(k,a)}\hat{p}_{(k,b)(k,c)}-\hat{p}_{(*,1)(k,b)}\hat{p}_{(k,a)(k,c)}\right), 
\end{eqnarray*}
where the first equality is a substitution using the definition of $p_e$ in the statement of the lemma, while the second comes from factoring to find a 3-1 binomial from $L(M_2)$, which is necessarily zero.

Now we consider the 3-1 and 2-1-1 quadratic equations in $L(M)$ with a single vertex in block $V_k$. Let $i,j\in[k-1]$, $a,b\in[n_i]$, and $c\in[n_j]$, and notice that by definition of $p_e$ from the lemma statement,
\begin{equation}\label{eqn:31211Binomials}
\begin{split}
p_{(i,a)(k,d)}p_{(i,b)(j,c)}-p_{(i,a)(j,c)}p_{(i,b)(k,d)} & = \left(\hat{p}_{(i,a)(*,1)}\hat{p}_{(*,1)(k,d)}P\right)\hat{p}_{(i,b)(j,c)}-\hat{p}_{(i,a)(j,c)}\left(\hat{p}_{(i,b)(*,1)}\hat{p}_{(*,1)(k,d)}P\right) \\
& =  \hat{p}_{(*,1)(k,d)}P\left(\hat{p}_{(i,a)(*,1)}\hat{p}_{(i,b)(j,c)}-\hat{p}_{(i,a)(j,c)}\hat{p}_{(i,b)(*,1)}\right). 
\end{split}
\end{equation}
When $i=j$ and $a,b,$ and $c$ are distinct, the left hand side of \cref{eqn:31211Binomials} is a 3-1 binomial with a single vertex in block $V_k$ from $L(M)$, and it must be 0 since the right hand side involves multiplication by a 3-1 binomial from $L(M_1)$. When $i\neq j$ and $a\neq b$, the left hand side of \cref{eqn:31211Binomials} is a 2-1-1 binomial with a single vertex in block $V_k$ from $L(M)$, and it must be 0 since the right hand side involves multiplication by a 2-1-1 binomial from $L(M_2)$.

Finally, consider the 2-1-1 and 2-2 quadratic equations from $L(M)$ with two vertices in block $V_k$. Let $i,j\in[k-1]$, $v\in[n_i]$, $u\in[n_j]$, and $a,b\in[n_k]$ with $a\neq b$ (if distinct $a$ and $b$ do not exist, there are no quadratic equations of these two types, and, thus, there is nothing to prove), and consider $p_{(i,v)(k,a)}p_{(j,u)(k,b)}-p_{(i,v)(k,b)}p_{(j,u)(k,a)}$. When $i=j$ and $v\neq u$, this is a 2-2 binomial with two vertices in block $V_k$ from $L(M)$ and when $i\neq j$, it is a 2-1-1 binomial with two vertices in block $V_k$. In both cases, one can see that it is zero by applying the definition of $p_e$ from the statement of the lemma and rearranging the first term to yield two identical terms. We have now shown that $\mathbf{p}$ satisfies all quadratic equations in $L(M)$. 

Now we will show that $A\mathbf{p}=A\mathbf{u}$. We will need the following:
First, let $i\in[k-1]$ and $v\in [n_i]$. Then by definition of $\mathbf{p}$, we have
\begin{equation}\label{eqn:psum1}
\sum_{w=1}^{n_k} p_{(i,v)(k,w)}=\sum_{w=1}^{n_k}\left(\hat{p}_{(i,v)(\ast,1)}\hat{p}_{(\ast,1)(k,w)}P\right)=\hat{p}_{(i,v)(\ast,1)}\left(\sum_{w=1}^{n_k}\hat{p}_{(\ast,1)(k,w)}\right)P=\hat{p}_{(i,v)(\ast,1)}P^{-1}P=\hat{p}_{(i,v)(\ast,1)}.
\end{equation} 
Similarly, for $w\in [n_k]$, the definition of $\mathbf{p}$ and \cref{lem:AnnoyingLittleEquation} yield 
\begin{equation}\label{eqn:psum2}
\sum_{i=1}^{k-1}\sum_{v=1}^{n_i} p_{(i,v)(k,w)}=\sum_{i=1}^{k-1}\sum_{v=1}^{n_i}\left(\hat{p}_{(i,v)(\ast,1)}\hat{p}_{(\ast,1)(k,w)}P\right)=\left(\sum_{i=1}^{k-1}\sum_{v=1}^{n_i}\hat{p}_{(i,v)(\ast,1)}\right)\hat{p}_{(\ast,1)(k,w)}P=\hat{p}_{(\ast,1)(k,w)}
\end{equation}
where the last equality holds because the expression in parentheses is $P^{-1}$ by \cref{lem:AnnoyingLittleEquation}.
Now we will show that $A\mathbf{p}=A\mathbf{u}$. First let $i\in [k-1]$ and $v\in[n_i]$. Since $\hat{\mathbf{p}}_1\in\mathcal{S}_1$, we must have $A_1\hat{\mathbf{p}}_1=A_1\mathbf{u}_1$. In particular, we must have the $\beta_{(i,v)}$ row of this equation (from $L(M_1)$), $\sum_{j=1}^{k-1}\sum_{u=1}^{n_j}\hat{p}_{(i,v)(j,u)}+\hat{p}_{(i,v)(\ast,1)}=\sum_{j=1}^{k-1}\sum_{u=1}^{n_j}u_{(i,v)(j,u)}+u_{(i,v)(\ast,1)}$. Applying \cref{eqn:psum1} along with the definitions of $u_{(i,v)(\ast,1)}$ and $\mathbf{p}$ yields
\[
\sum_{j=1}^{k-1}\sum_{u=1}^{n_j}{p}_{(i,v)(j,u)}+\sum_{w=1}^{n_k}p_{(i,v)(k,w)}=\sum_{j=1}^{k-1}\sum_{u=1}^{n_j}u_{(i,v)(j,u)}+\sum_{w=1}^{n_k}u_{(i,v)(k,w)},
\]
which is the $\beta_{(i,v)}$ row of $A\mathbf{p}=A\mathbf{u}$. Now let $w\in[n_k]$. Then since $\hat{\mathbf{p}}_2\in\mathcal{S}_2$, we must have $A_2\hat{\mathbf{p}}_2=A_2\mathbf{u}_2$, so we must have the $\beta_{(k,w)}$ row of this equation (from $L(M_2)$), $\hat{p}_{(\ast,1)(k,w)}+\sum_{a=1}^{n_k}\hat{p}_{(k,a)(k,w)}=u_{(\ast,1)(k,w)}+\sum_{a=1}^{n_k}u_{(k,a)(k,w)}$. By applying \cref{eqn:psum2} and the definitions of $u_{(\ast,1)(k,w)}$ and $\mathbf{p}$, we obtain the $\beta_{(k,w)}$ row of $A\mathbf{p}=A\mathbf{u}$,
\[
\sum_{i=1}^{k-1}\sum_{v=1}^{n_i}p_{(i,v)(k,w)}+\sum_{a=1}^{n_k}p_{(k,a)(k,w)}=\sum_{i=1}^{k-1}\sum_{v=1}^{n_i}u_{(i,v)(k,w)}+\sum_{a=1}^{n_k}u_{(k,a)(k,w)}.
\]
We now have all of the $\beta$ rows of $A\mathbf{p}=A\mathbf{u}$ and just need to address the $\alpha$ rows. First, let $i,j\in[k-1]$ and notice that $\mathbf{p}$ must satisfy the $\alpha_{i,j}$ row of $A\mathbf{p}=A\mathbf{u}$ since this equation is the same as the equation arising from the $\alpha_{i,j}$ row of $A_1\hat{\mathbf{p}}_1=A_1\mathbf{u}_1$ once we use that $p_{(i,v)(j,w)}=\hat{p}_{(i,v)(j,w)}$ for all $v\in n_i$, $w\in n_j$. Similarly, the $\mathbf{p}$ must satisfy the $\alpha_{k,k}$ row of $A\mathbf{p}=A\mathbf{u}$ since this equation is the same as the $\alpha_{k,k}$ row of $A_2\hat{\mathbf{p}}_2=A_2\mathbf{u}_2$ since $p_{(k,a)(k,b)}=\hat{p}_{(k,a)(k,b)}$ for all $a,b\in[n_k]$. 

Let $i\in[k-1]$. We must show $\mathbf{p}$ satisfies the $\alpha_{i,k}$ row of $A\mathbf{p}=A\mathbf{u}$. The $\alpha_{i,\ast}$ row of $A_1\hat{\mathbf{p}}_1=A_1\mathbf{u}_1$ is $\sum_{v=1}^{n_i} \hat{p}_{(i,v)(\ast,1)}=\sum_{v=1}^{n_i} u_{(i,v)(\ast,1)}$. \cref{eqn:psum1} and the definition of $u_{(i,v)(\ast,1)}$ yields $\sum_{v=1}^{n_i}\sum_{w=1}^{n_k}p_{(i,v)(k,w)}\\=
\sum_{v=1}^{n_i}\sum_{w=1}^{n_k}u_{(i,v)(k,w)}$, which is the $\alpha_{i,k}$ row of $A\mathbf{p}=A\mathbf{u}$. Thus, we've shown that $\mathbf{p}\in\mathcal{S}$.
\end{proof}

\begin{lemma}\label{lem:phiBijection}
The map $\phi:\mathcal{S}\rightarrow\mathcal{S}_1\times\mathcal{S}_2$ defined in \cref{def:phidef} is a bijection. 
\end{lemma}

\begin{proof}
Let $\mathbf{u}\in \mathbb{C}^{|E|}$ and define $\mathbf{u}_1\in \mathbb{C}^{|E_1|}$ and $\mathbf{u}_2\in\mathbb{C}^{|E_2|}$ as in Definitions \ref{def:M1CollapseLast} and \ref{def:M2CollapseFirst}. Define $\phi$ as in \cref{def:phidef}. Let $\mathbf{p},\mathbf{q}\in\mathcal{S}$, and assume $(\mathbf{p}_1,\mathbf{p}_2)=\phi(\mathbf{p})=\phi(\mathbf{q})=(\mathbf{q}_1,\mathbf{q}_2)$. Since $\mathbf{p}_1=\mathbf{q}_1$ and $\mathbf{p}_2=\mathbf{q}_2$, $p_e=q_e$ for all $e\in E_1 \cup E_2$. Thus, it remains to show that $p_e=q_e$ for $e\in E\setminus (E_1\cup E_2)$. Let $i\in[k-1]$, $v\in [n_i]$, $w\in[n_k]$, and let $e=(i,v)(k,w)\in E_{ik}$. Since $\mathbf{p},\mathbf{q}\in\mathcal{S}$, both must satisfy all quadratic equations of $L(M)$, and we will argue that this implies that they must satisfy
\begin{equation}\label{eqn:pBinomial}
p_{(i,v)(k,w)}p_{(b,c)(k,a)}-p_{(i,v)(k,a)}p_{(b,c)(k,w)}=0
\end{equation}
\begin{equation}\label{eqn:qBinomial}
q_{(i,v)(k,w)}q_{(b,c)(k,a)}-q_{(i,v)(k,a)}q_{(b,c)(k,w)}=0
\end{equation}
for all $b\in[k-1]$, $c\in[n_b]$, and $a\in[n_k]$. 

Let $k=2$. Then  $i=b=1$, so when $a\neq w$ and $c\neq v$, \cref{eqn:pBinomial} and \cref{eqn:qBinomial} are 2-2 quadratic equations of $L(M)$ and must hold. If $a=w$ or $c=v$ then each binomial is identically 0, and so the $k=2$ case is complete.

Let $k>2$. If $a\neq w$ and $b\neq i$, \cref{eqn:pBinomial} and \cref{eqn:qBinomial} are 2-1-1 quadratic equations of $L(M)$ and must hold. When $a\neq w$, $b=i$, and $c\neq v$, these equations become 2-2 quadratic equations of $L(M)$ and must hold. When $a=w$, or both $b=i$ and $c=v$, then each binomial is identically 0, and Equations \ref{eqn:pBinomial} and \ref{eqn:qBinomial} hold when $k>2$ for all $b\in[k-1]$, $c\in[n_b]$, and $a\in[n_k]$. Therefore, we can sum over $a, b$, and $c$ to obtain
\[
\sum_{a=1}^{n_k}\left(\sum_{b=1}^{k-1}\left(\sum_{c=1}^{n_b}\left(p_{(i,v)(k,w)}p_{(b,c)(k,a)}-p_{(i,v)(k,a)}p_{(b,c)(k,w)}\right)\right)\right)=0
\]
\[
\Longrightarrow p_{(i,v)(k,w)}\left(\sum_{a=1}^{n_k}\left(\sum_{b=1}^{k-1}\left(\sum_{c=1}^{n_b}p_{(b,c)(k,a)}\right)\right)\right)-\left(\sum_{a=1}^{n_k} p_{(i,v)(k,a)}\right)\left(\sum_{b=1}^{k-1}\left(\sum_{c=1}^{n_b} p_{(b,c)(k,w)}\right)\right)=0
\]
\[
\Longrightarrow p_{(i,v)(k,w)}\left(\sum_{a=1}^{n_k}p_{(*,1)(k,a)}\right)-p_{(i,v)(*,1)}p_{(*,1)(k,w)}=0
\Longrightarrow p_{(i,v)(k,w)}=\frac{p_{(i,v)(*,1)}p_{(*,1)(k,w)}}{\sum_{a=1}^{n_k}p_{(*,1)(k,a)}},
\]
where the denominator is non-zero for generic $\mathbf{u}$. The same computation with \cref{eqn:qBinomial} yields
\[
p_{(i,v)(k,w)}=\frac{p_{(i,v)(*,1)}p_{(*,1)(k,w)}}{\sum_{a=1}^{n_k}p_{(*,1)(k,a)}}=\frac{q_{(i,v)(*,1)}q_{(*,1)(k,w)}}{\sum_{a=1}^{n_k}q_{(*,1)(k,a)}}=q_{(i,v)(k,w)}
\]
where the center equality holds because all indices involved in the two expressions are in $E_{* k}\cup \left( \bigcup_{i=1}^{k-1} E_{i*} \right) \subseteq E_1\cup E_2$ and we know $p_e=q_e$ for $e\in E_1\cup E_2$. Thus, $p_e=q_e$ for all $e\in  \bigcup_{i=1}^{k-1} E_{ik}$. Since $E \subseteq E_1\cup E_2 \cup \left( \bigcup_{i=1}^{k-1} E_{ik} \right)$, $p_e=q_e$ for all $e\in E$. Thus, $\mathbf{p}=\mathbf{q}$, and $\phi$ is injective.

Now we show $\phi$ is surjective. Let $(\mathbf{\hat{p}}_1,\mathbf{\hat{p}}_2)\in \mathcal{S}_1\times\mathcal{S}_2$. Consider the vector $\mathbf{p}\in\mathcal{S}$ as defined in \cref{lem:Backwards}. By definition of $\mathbf{p}$, $p_e=\hat{p}_e$ for $e\in (E_1\setminus \left(\bigcup_{i=1}^{k-1}E_{i*}\right))\cup(E_2\setminus E_{* k})$. Therefore, for coordinates $p_e$ such that $e\in (E_1\setminus \left(\bigcup_{i=1}^{k-1}E_{i*}\right))$, $\mathbf{p}_1=\mathbf{\hat{p}}_1$ and for coordinates $p_e$ such that $e\in (E_2\setminus E_{* k})$, $\mathbf{p}_2=\mathbf{\hat{p}}_2$. 

Now suppose $e\in \left(\bigcup_{i=1}^{k-1}E_{i*}\right)$, that is, $e=(i,v)(*,1)$ for some $i\in[k-1], v\in [n_i]$. Then by using \cref{eqn:psum1} and the definition of $\phi$, we see that the coordinate of $\mathbf{p}_1$ indexed by $e$ is $p_e=p_{(i,v)(*,1)}=\sum_{a=1}^{n_k} p_{(i,v)(k,a)}=\hat{p}_{(i,v)(\ast,1)}=\hat{p}_e$, so $\mathbf{p}_1=\mathbf{\hat{p}}_1$. Now let $e\in E_{* k}$, that is, $e=(*,1)(k,w)$ for some $w\in[n_k]$. A similar argument as above, using \cref{eqn:psum2} shows $p_e=\hat{p}_e$, so $\mathbf{p}_2=\mathbf{\hat{p}}_2$. Thus, $\phi(\mathbf{p})=(\mathbf{p}_1,\mathbf{p}_2)=(\mathbf{\hat{p}}_1,\mathbf{\hat{p}}_2)$, so $\phi$ is a surjection and hence, a bijection.
\end{proof} 

Now we are ready to prove the Factoring Lemma.

\begin{proof}[Proof of \cref{lem:FactoringLemma} (Factoring Lemma):]\,

Consider the $\beta$-SBM $\mathcal{M}(n_1,n_2,\dots,n_k)$ such that $k>1$. By \cref{lem:phiBijection} the map $\phi:\mathcal{S}\rightarrow\mathcal{S}_1\times\mathcal{S}_2$ defined in \cref{def:phidef} is a bijection. Thus, $|\mathcal{S}|=|\mathcal{S}_1||\mathcal{S}_2|$, so 
\[
\MLdeg(n_1,n_2,\dots,n_k)=\MLdeg(n_1,n_2,\dots,n_{k-1},1)\MLdeg(n_k,1).
\]
\end{proof}

\section{Acknowledgements}
This work is a product of the authors' participation in the 2023 Research Experiences for Undergraduate Faculty (REUF) program, hosted and funded by the Institute for Computational and Experimental Research in Mathematics (ICERM) with support from National Science Foundation grant DMS-2015375. The authors received additional support through a 2024 REUF continuation workshop at the American Institute of Mathematics (AIM), funded by NSF grant DMS-2015462. EG was supported by the National Science Foundation grant DMS-1945584. We thank Scott Greenhalgh for helpful discussions.

\bibliographystyle{plainurl}
\bibliography{refs}

\end{document}